\documentclass[leqno,11pt]{amsart}
\usepackage[colorlinks,pagebackref,hypertexnames=false]{hyperref}

\usepackage{amsmath,amsthm,amssymb,mathrsfs} 
\usepackage[alphabetic,backrefs]{amsrefs}
\usepackage{ae,aecompl}
\usepackage[margin=1.2in]{geometry} 

\usepackage{cite}

\usepackage[english]{babel}

\usepackage{bookmark}

\numberwithin{equation}{section}
\usepackage{microtype}

\usepackage{fancybox,fancyhdr,graphics,epsfig}

\numberwithin{equation}{section}

\newcommand{\cE}{\mathcal{E}}

\newcommand{\SO}{{\rm SO}}
\newcommand{\G}{{\rm G}}
\newcommand{\Sp}{{\rm Sp}}
\newcommand{\Spin}{{\rm Spin}}
\newcommand{\SU}{{\rm SU}}

\renewcommand{\epsilon}{\varepsilon}

\newcommand{\Hom}{{\mathrm{Hom}}}

\newcommand{\dvol}{\mathop\mathrm{dvol}\nolimits}
\newcommand{\id}{\mathrm{id}}

\renewcommand{\Re}{\mathop{\mathrm{Re}}}

\newcommand{\tr}{\mathop{\mathrm{tr}}\nolimits}

\newcommand{\Fr}{{\rm Fr}}

\def\<{\mathopen{}\left<}
\def\>{\right>\mathclose{}}
\def\({\mathopen{}\left(}
\def\){\right)\mathclose{}}

\newtheorem{Mtheorem}{Main Theorem}
\newtheorem{theorem}{Theorem}

\newtheorem{corollary}[theorem]{Corollary}

\newtheorem{example}[theorem]{Example}

\newtheorem{lemma}[theorem]{Lemma}

\newtheorem{proposition}[theorem]{Proposition}
\newtheorem{remark}[theorem]{Remark}

\numberwithin{equation}{section}

\newtheorem*{theorem*}{Theorem}
\newtheorem*{prop*}{Proposition}
\newtheorem*{corol*}{Corollary}

\author{Andrew Clarke}
\address{Instituto de Matem\'atica, Universidade Federal do Rio de Janeiro, Av. Athos da Silveira Ramos 149, Rio de Janeiro, RJ, 21941-909, Brazil}
\email{andrew@im.ufrj.br}

\author{Gon\c calo Oliveira}
\address{Universidade Federal Fluminense, IME-GMA, Niter\'oi, Brazil}
\email{galato97@gmail.com}

\title{$\Spin(7)$-instantons from evolution equations}

\date{\today}

\begin{document}
\maketitle

\begin{abstract}
	
	In this paper we study $\Spin(7)$-instantons on asymptotically conical $\Spin(7)$-orbifolds (and manifolds) obtained by filling in certain squashed $3$-Sasakian $7$-manifolds.
We construct a $1$-parameter family of explicit $\Spin(7)$-instantons. Taking the parameter to infinity, the family (a) bubbles off an ASD connection in directions transverse to a certain Cayley submanifold $Z$, (b) away from $Z$ smoothly converges to a limit $\Spin(7)$-instanton that extends across $Z$ onto a topologically distinct bundle, (c) satisfies an energy conservation law for the instantons and the bubbles concentrated on $Z$, and (d) determines a Fueter section, in the sense of \cites{Donaldson2009,Haydys2011,Walpuski2017}.
	
\end{abstract}

\tableofcontents




\section{Introduction}

In this article we study gauge theory on $8$-dimensional Riemannian manifolds whose holonomy is contained in the group $\Spin(7)$. This is to say, $\Spin(7)$-instantons.

\subsection*{$\Spin(7)$-instantons}




A $\Spin(7)$-manifold $M$ is a Riemannian manifold whose holonomy is the Lie group $\Spin(7)$. Together with $\G_2$, this is one of the exceptional cases in Berger's classification of the possible Riemannian holonomy groups \cites{Berger1955} for non-symmetric irreducible simply connected Riemannian manifolds. A $\Spin(7)$-manifold comes equipped with a certain closed $4$-form $\Theta$ known as the Cayley form. It turns out that $\Theta$ is actually a calibration and so the submanifolds $N\subseteq M$ satisfying $\Theta |_N = \dvol_{g|_N}$ minimize volume in their homology class. Such manifolds are known as Cayley submanifolds and have both known and expected relations with certain connections known as $\Spin(7)$-instantons. Let $\G$ be a compact semisimple Lie group, a connection $A$ on a $\G$-bundle is said to be a $\G_2$-instanton if its curvature $F_A$ satisfies
\begin{eqnarray}\label{eq:Spin7Inst}
*F_A =-F_A \wedge \Theta,
\end{eqnarray}
which suggests similarities with the theory of anti-self-dual connections in $4$-dimensions. The questions that we consider in this paper arise from a reduction of the partial differential equation $d\Theta=0$ to a system of ordinary differential equations, and the instantons that can be obtained by a reduction of Equation (\ref{eq:Spin7Inst}) to a system of ODE's.

The chronology of the study of $\Spin(7)$-manifolds is relatively well known. The classification of Riemannian holonomy groups of non-symmetric, irreducible, and simply connected Riemannian manifolds was made by Berger \cites{Berger1955} and then Simons \cites{Simons1962}, after which Bonan \cites{Bonan1966} showed that the $\Spin(7)$-holonomy condition was equivalent to the existence of a parallel $4$-form. Fernandez \cites{Fernandez86} then proved it was sufficient that the form be closed to have the holonomy reduction. However, it was only 32 years after Berger's classification of the possible Riemannian holonomy groups that in 1987 Bryant \cites{Bryant1987} constructed the first manifolds with $\Spin(7)$-holonomy. Soon after this, the first complete example was found by Bryant and Salamon in \cites{BS89}. This example constitutes the main point of departure of our work in this paper. As forcompact manifolds with $\Spin(7)$-holonomy, these were first given by Joyce in \cites{Joyce1999}.

The history of gauge theory on $\Spin(7)$-manifolds is less well-known. The first-order instanton equation in higher dimensions appeared in the physics literature over 35 years ago \cites{Corrigan1983,Ward1984}. In the mathematics literature on gauge theory in higher dimensions, the foundational references that suggest many geometric motivations, and show similarities and differences with the $4$-dimensional case remain as \cites{Donaldson1998,Donaldson2009,Tian2000,Tao2004}. The study of $\Spin(7)$-instantons on the compact manifolds constructed by Joyce was initiated by  \cites{Lewis1998} and was later clarified and extended by Walpuski \cites{Walpuski2017_Spin} and Tanaka \cites{Tanaka2012}. Nontrivial examples on complete noncompact manifolds were however still unknown until our work in this paper.
 The first results in this direction were given in the first named author's work in \cites{Clarke14}, where instantons with singularities are constructed. In \cites{Lotay2017}, the instanton equation is regarded as an evolution equation and under symmetry assumptions the authors reduce the instanton equation to an ODE. This ODE was then integrated yielding a family of solutions which also develop singularities and may be related to those found earlier by the first named author. We also note the article of Harland and N\"olle \cites{Harland2012}, which studies instantons in the presence of Killing spinors, and on the associated cone manifolds.

\subsection*{Summary of the main results}

Let $Z$ be an Einstein $4$-orbifold with positive scalar curvature and anti-self-dual Weyl curvature. Furthermore, suppose that $Z$ is spin, for instance if it is smooth then it must be $\mathbb{S}^4$. Then, the total space of its bundle of positive chirality spinors, which we shall denote by $\pi: E \rightarrow Z$, admits the $\Spin(7)$-holonomy metric constructed by Bryant-Salamon \cites{BS89}. In particular, the zero section $Z \subset E$ is the unique compact Cayley submanifold of the Bryant-Salamon manifold.

We shall now state a non-technical version of our main results regarding $\Spin(7)$-instantons with gauge group $\SU(2)$. In order to write such a non-technical statement we have decided to omit several other results and we encourage the reader to find these in the next subsection which outlines the results and proofs in this paper. In the statement below we shall regard the instantons as connections on the complex rank $2$ vector bundles associated with respect to the standard $\SU(2)$-representation.

\begin{Mtheorem}
There is a real $1$-parameter family of $\Spin(7)$-instantons $\lbrace A_{y_0} \rbrace_{y_0 \in [0,+\infty)}$ on the trivial bundle $\mathbb{C}^2 \times E \rightarrow E$ which are irreducible for $y_0 \neq 0$. As $y_0 \rightarrow + \infty$, the instantons $A_{y_0}$ converge uniformly, with all derivatives, on compact subsets of $E \backslash Z$ to an instanton $A_{\lim}$. Furthermore,
\begin{enumerate}
	\item[(a)] $A_{\lim}$ extends smoothly across $Z$, and yields a $\Spin(7)$-instanton on the bundle $\pi^* E \rightarrow E$.
	\item[(b)] As $y_0 \rightarrow + \infty$, the $\Spin(7)$-instantons bubble off a charge-$1$ anti-self-dual connection transversely to the Cayley $Z\subset E$.
	\item[(c)] Regarding the energy densities $| F_{A_{y_0}} |^2$, $| F_{A_{\lim}} |^2$ as currents, as $y_0 \rightarrow + \infty$ 
	$$| F_{A_{y_0}} |^2 \rightharpoonup | F_{A_{\lim}} |^2 + 8 \pi^2 \delta_{Z},$$
	where $\delta_{Z}$ is the 
	current of integration associated with the Cayley submanifold $Z \subset E$.
\end{enumerate}
\end{Mtheorem}

These results may be interpreted as follows. Part (a) in the statement above illustrates two different phenomena: an example of a removable singularity phenomenon, and of a bundle change. Indeed, we have a sequence of $\Spin(7)$-instantons converging to a limiting instanton outside a calibrated submanifold $Z$. The limit connection extends across $Z$, although it does so on a topologically distinct bundle. Parts (b) and (c) state that as $y_0 \rightarrow + \infty$ the $\Spin(7)$-instantons concentrate along the calibrated submanifold $Z \subset E$ and as stated in (b), bubble anti-self-dual connections on the $\mathbb{R}^4$'s transverse to $Z$. We may also interpret (c) as an energy conservation law, indeed $8\pi^2$ is precisely the energy of the charge-1 anti-self-dual connection on $\mathbb{R}^4$, i.e. the charge of the bubble.

The $\Spin(7)$-instanton equation implies that the family of anti-self-dual connections bubbling of transversely to $Z$ must satisfy an equation known as the Fueter equation. This is a quaternionic analogue of the Cauchy-Riemann equations for a function from $Z$ to the moduli space of anti-self-dual connections in the normal directions to $Z$. From our work we explicitly obtain a solution of this equation which interpreted in the correct way is simply a certain identity map, see Section \ref{subsec:Fueter}. For more information on the Fueter equation in this context, see \cites{Walpuski2017_Spin}.

\subsection*{Outline of the proof}

Consider $E \backslash Z$, the complement of the zero section $Z$ in $E$. Then, $E \backslash Z \cong (0,\infty)_r \times X$, where $X \rightarrow Z$ is the total space of the $\SU(2)$-bundle with respect to which $E$ is associated to the standard representation of $\SU(2)$ on $\mathbb{C}^2$. In this polar type decomposition of $E \backslash Z$ we can write $\Theta=dr\wedge \varphi+\psi$ for $\varphi=\varphi(r)$ and $\psi=*_\varphi\varphi$ depending on $r\in (0,\infty)$. The explicit way in which this happens is given in Section \ref{sec:Structure} and depends on two functions $t(r)$ and $\nu(r)$ of the radial coordinate $r$. The condition for $\Theta$ to be closed determines a non-linear system of ordinary differential equations for $t$ and $\nu$. We write solutions $(t(r),\nu(r))$ of this system as integral curves of a $1$-form on $\mathbb{R}^2$. By a homogeneity argument, we show that initial conditions at $r=0$ can always be found so that complete solutions exist and give rise to a unique $\Spin(7)$-holonomy metric up to scaling. 
The resulting metric is always asymptotically conical and for $X=S^7$, i.e. for $Z= \mathbb{S}^4$, we recover the metric of Bryant and Salamon \cites{BS89} in the form given by \cites{Gibbons1990}.          

On the trivial $\G$-bundle over $X$, connections are determined by $\mathfrak{g}$-valued $1$-forms on $X$. We consider the trivial $\SU(2)$-bundle over $E \backslash Z$ equipped with a family of connections depending on a function $x(r)$ of the radial coordinate $r$. Then, the condition that the resulting connection be a $\Spin(7)$-instanton turns into an ordinary differential equation for $x(r)$. The first task at hand now is to determine initial conditions for $x(r)$ at $r=0$ so that the resulting $\Spin(7)$-instantons smoothly extend across $Z \subset E$. As an application of the removable singularity theorem of Tao and Tian we show that such $\Spin(7)$-instantons extend across the zero section, up to gauge, if and only if their curvature remains bounded in a neighbourhood of $Z$. This then leads to a dichotomy on the possible values of $x(0)$ and thus to two different singular initial value problems for the function $x(r)$. We solve these using a refinement, due to Malgrange, of the fundamental existence and uniqueness theorem for solutions to ordinary differential equations. It turns out that the different singular initial value problems arising from the dichotomy mentioned above are related to the different topological types of the bundles on which the resulting $\Spin(7)$-instantons extend across $Z$. The precise result we obtain is as follows.

\begin{theorem*}
Let $A$ be a smooth $\Spin(7)$-instanton for $\Theta$ on $E$ arising from the procedure above. Then, either
\begin{enumerate}
\item $A|_Z$ is the connection on $E$ induced by the Levi-Civita connection of the anti-self-dual, Einstein metric on $Z$. In this case there is a unique $\Spin(7)$-instanton $A_{\lim}$ on $\pi^*E \rightarrow E$ with this property.
\item $A|_Z$ is the flat connection on the trivial $\mathbb{C}^2$-bundle over $Z$, in which case there is a real $1$-parameter family of such $\Spin(7)$-instantons. These live on the trivial bundle $E \times \mathbb{C}^2$, are parametrized by $y_0 \in [0, + \infty)$ with $A_0$ being flat and the $A_{y_0}$ being irreducible for $y_0 >0$.
 \end{enumerate}
\end{theorem*}

Our next theorem exhibits a phenomenon that also appears in \cites{Lotay2018} in the context of invariant $\G_2$-instantons. This theorem gives geometric meaning to the constant $y_0$ that parametrizes our family of $\Spin(7)$-instantons, as they appear in the previous theorem. For this, we need to consider the standard charge-$1$ anti-self-dual instantons on $\mathbb{R}^4$, with respect to the Euclidean metric. Choosing the origin as the centre of charge concentration, these appear in a family $\{A^{asd}_\kappa\}$ for $\kappa>0$. At a point $z\in Z$, denote by $N_zZ\cong \mathbb{R}^4$ (with induced metric) the normal space at $z$ to the submanifold $Z\subseteq E$.

\begin{theorem*}
	Let $\lbrace A_{y_0} \rbrace_{y_0}$ be a sequence of $\Spin(7)$-instantons given in the previous theorem with $y_0\to \infty$. Then, the following are true:
	\begin{itemize}
		\item[(a)] Given any $\kappa>0$, there is a sequence $0<\lambda=\lambda(y_0,\kappa)\to 0$ with the following significance: for all $z \in Z$, $(s_{z,\lambda})^* A_{y_0}$ uniformly converges (with all derivatives) to $A^{asd}_{\kappa}$ on $B_1\subseteq N_zZ$. 
		\item[(b)] The connections $A_{y_0}$ uniformly converge (with all derivatives) to $A_{\lim}$ on all compact subsets of $E \backslash Z $. 
	\end{itemize}
	\end{theorem*}
Here, the maps $s_{z,\lambda}:B(0,1)\to B_\lambda(z)$ are given by the composition of dilation on the normal bundle with the Riemannian exponential map. 
We note here how this result relates the two families of connections $A_{\lim}$ and $\{A_{y_0}\}$ that appear in previous theorem. Indeed, this can be interpreted as saying that in order to compactify the moduli space of instantons on the trivial bundle $\mathbb{C}^2 \times E \rightarrow E$ we must add (at least one) instanton on the topologically inequivalent bundle $\pi^* E \rightarrow E$. 

In a sense, this indicates that the curvatures of the sequence of connections $\{ A_{y_0} \}$, in directions normal to $Z$, concentrate along $Z$ in the same way that the standard $4$-dimensional instanton curvature  concentrates at the origin, as $\kappa\to 0$. This heuristic statement can be made more precise by our final main theorem, which indicates a form of energy conservation between the connections $\{ A_{y_0} \}$, $A_{\lim}$ and $A^{asd}_\kappa$. 

\begin{theorem*}\label{thm:delta}
	Let $y_0>0$, then the function $| F_{A_{y_0}} |^2 - | F_{A_{\lim}} |^2$ is integrable. Moreover, as $y_0 \nearrow + \infty$ this functions converge as currents to $8 \pi^2 \delta_{Z}$, i.e. for all compactly supported $f \in C^{\infty}_0(E, \mathbb{R})$ 
	$$\lim_{y_0 \rightarrow + \infty} \int_{E} f  (| F_{A_{y_0}} |^2 - \vert F_{A_{\lim}} |^2)  = 8 \pi^2 \int_{Z} f ,$$
	where in the later integral we use the induced metric on the Cayley submanifold $Z \subset E$.
\end{theorem*}
We first note that the coefficient of $8\pi^2$ on the right is exactly the total energy of the charge-$1$ instanton on $\mathbb{R}^4$. Next, we observe that this result is similar in spirit to the discussion of Tian \cites{Tian2000} on the compactification of the moduli space; however, that result does not directly apply in this case as the total energy is not finite.\

There are many obvious similarities between the results of this paper and those of a previous work of the second author with Jason Lotay \cites{Lotay2018} that considers gauge theory on certain non-compact $G_2$-manifolds. The principal difference is that here we do not presuppose any homogeneity properties of the transverse slice $X$. That hypothesis is made redundent in this case by the $3$-Sasakian condition, which furnishes the ingredients for our construction. Also, there already exist a large number of cohomogeneity-$1$ $G_2$ manifolds, thereby motivating a more detailed treatment according to their isotropy groups. Complete $\Spin(7)$ manifolds are more sparce, and our treatment, as far as we know, already describes all known examples of asymptotically conical $\Spin(7)$-manifolds.

\subsubsection*{Acknowledgments}

The second named author wants to thank Alex Waldron and Thomas Walpuski for their comments.

\section{Prerequisites from $\Spin(7)$ geometry}

\subsection{$\Spin(7)$-manifolds}
The standard definition of the group $\Spin(7)$ is as the universal (double) cover of $\SO(7)$. However, via the spin representation, we can consider it as a subgroup $\Spin(7) \subset \SO(8)$. By the Berger-Simons classification of Riemannian holonomy groups, $\Spin(7)$, with this representation, is one of the candidates as an exceptional holonomy group. In fact, by a theorem of Fernandez \cites{Fernandez86} the reduction of holonomy group to $\Spin(7)$ is equivalent to the existence of a closed, $4$-form $\Theta$ such that at each $p\in M$, $\Theta_p$ takes values in a certain closed submanifold of $\Lambda^4T_p^*M$. This is determined by the condition that there exists an isomorphism between $T_pM$ and $\mathbb{R}^8$ identifying $\Theta_p$ with the $4$-form\footnote{A perhaps more transparent expression for this form is as $\Theta_0=dx^0\wedge \varphi_0+\psi_0$ where $\varphi_0$ and $\psi_0 $ are the fundamental $3$ and $4$ forms on $\mathbb{R}^7$ that are invariant with respect to the subgroup $\G_2\subseteq \Spin(7)$. }
\begin{eqnarray*}
	\Theta_0 &=& dx^{0123}- dx^{0145}-dx^{0167} -dx^{0246}-dx^{0275}-dx^{0347}-dx^{0356}\\
	&&\ \ \ \ +dx^{4567} -dx^{1247}-dx^{1256}-dx^{2345}-dx^{2367}-dx^{3146}-dx^{3175}.
\end{eqnarray*} 
The condition that at all points $\Theta$ takes values in this closed submanifold yields a reduction of the principal frame bundle to have structure group $\Spin(7)={\rm Stab(\Theta_p)}$. For a given Riemannian metric, the condition that the holonomy group be contained in $\Spin(7)$ is equivalent to the existence of some $4$-form $\Theta$, as above, for which the reduced principal bundle is preserved by the Levi-Civit\`a connection of the metric. The theorem of Fern\'andez says that this in turn is equivalent to such a $\Theta$ existing, but only with the condition $d\Theta=0$.
This is usually summarized by saying that the $\Spin(7)$-structure determined by $\Theta$ is {\sl torsion-free}.

Having the previous discussion in mind, we shall regard a $\Spin(7)$-manifold to be a triple $(M,g,\Theta)$ as in the previous paragraph. In such a manifold, the reduction of structure group to $\Spin(7)$ determines decompositions of the bundles $\Lambda^kT^* M$ into sub-bundles corresponding to irreducible representations of $\Spin(7)$ on $\Lambda^k(\mathbb{R}^8)^*$. In particular, $\Lambda^2=\Lambda^2_7 \oplus \Lambda^2_{21}$ where $\Lambda^2_7 \cong \mathbb{R}^7$ is the standard representation of $\Spin(7)$ via $\SO(7)$, and $\Lambda^2_{21}\cong \mathfrak{spin}(7)\subseteq \mathfrak{so}(8)\cong \Lambda^2$ with the adjoint representation of $\Spin(7)$ on its Lie algebra $\mathfrak{spin}(7)$.\\

\subsection{$\Spin(7)$-instantons}
The gauge theoretic conditions that we consider are the following. Let $\G$ be a compact connected Lie group and $P\to M$ a principal $\G$-bundle on the $\Spin(7)$-manifold $M$. The curvature, $F_A$, of a connection $A$ on $P$ is a section of the vector bundle $\Lambda^2\otimes \mathfrak{g}_P$. We say that $A$ is a $\Spin(7)$ instanton if at each point $F_A$ takes values in the sub-bundle $\Lambda^2_{21}\otimes \mathfrak{g}_P$. This is equivalent to $A$ satisfying the equation
\begin{eqnarray}\label{eq:Spin7Inst_2}
*F_A=-F_A\wedge \Theta.
\end{eqnarray}
As remarked in the introduction, this equation has similarities with $4$ dimensional anti-self-dual equations. $\Spin(7)$-instantons are example of Yang-Mills connections which are the critical points of the Yang-Mills function
$$\cE(A)= \int_{M} |F_A|^2 \dvol ,$$
with respect to compactly supported variations. This makes little sense in our case as the examples that we construct have infinite energy,
 but even in this case the Yang-Mills equation $d_A \ast F_A=0$ can be considered. Using \ref{eq:Spin7Inst_2} we immediately see that $\Spin(7)$-instantons are Yang-Mills as
$$d_A \ast F_A = - d_A F_A \wedge \Theta - F_A \wedge d \Theta =0,$$
by the Bianchi identity $d_AF_A=0$ and the fact that $\Theta$ is closed.

\section{AC fillings of squashed $3$-Sasakians}\label{sec:Structure}

On a $7$-dimensional spin manifold $X^7$ consider a $\G_2$-structure whose associated $3$ and $4$-forms are $\varphi,\psi=*_\varphi\varphi$. Then, on $(0,\infty)_r\times X$, the $4$-form $\Theta=r^3dr\wedge\varphi+r^4\psi$ determines a reduction of the structure group to $\Spin(7)$ whose associated metric is the cone over $g_\varphi$. If the metric on $X$ is {\sl $3$-Sasakian}, then the cone metric is hyperK\"ahler, and has holonomy group contained in $\Sp(2)\subseteq \Spin(7)$. However, in the case where $X$ supports a $3$-Sasakian structure, there exists another metric obtained by shrinking three of the seven dimensions by a common factor such that the new cone metric has holonomy equal to $\Spin(7)$. We shall now review this construction, use it to re-derive the Bryant-Salamon $\Spin(7)$-manifold in a way compatible with writing the family of connections we are interested in studying.

\subsection*{Preparation from $3$-Sasakian geometry}

A $3$-Sasakian $7$-manifold is a Riemannian $7$-manifold $(X^7, g_7)$ equipped with $3$-orthonormal Killing vector fields $\lbrace \xi_i \rbrace_{i=1}^3$ satisfying $[\xi_i, \xi_j] = \epsilon_{ijk} \xi_k$. Any $3$-Sasakian $X$ is quasi-regular in the sense that the vector fields $\lbrace \xi_i \rbrace_{i=1}^3$ generate a locally free $\SU(2)$ action. The space of leaves $Z^4$ comes equipped with the Riemannian metric $g_Z$ such that $\pi:X \rightarrow Z$ is an orbifold Riemannian submersion. With respect to the metric $g_Z$, $Z$ has the structure of an anti-self-dual Einstein orbifold with scalar curvature $s>0$. If $Z$ is spin, then we can regard $\pi: X\rightarrow Z$ as the lift to $\SU(2)$ of an $\SO(3)$, (orbi)-bundle of frames of $\Lambda^2_+ Z$. The Levi-Civita connection of $Z$ equips it with a connection $\eta=\eta_i \otimes T_i \in \Omega^1(X^7, \mathfrak{su}(2))$, where the $T_i$ are a standard basis of $\mathfrak{su}(2)$ satisfying $[T_i,T_j]=2\epsilon_{ijk}T_k$. This has the property that the $\eta$-horizontal forms $\omega_i$ defined by
$$F_{\eta}= d\eta + \frac{1}{2}[\eta \wedge \eta] = - \frac{s}{24} \  \omega_i \otimes T_i,$$
form an orthogonal basis of $(\Lambda^2_+ \ker (\eta), g_7 \vert_{\ker (\eta)})$ with $\vert \omega_i \vert = \sqrt{2}$ and $s \in \mathbb{R}^+$. We further remark that the $3$-Sasakian metric $g_7$ can be written as
\begin{eqnarray*}
g^{ts} = \eta^i \otimes \eta^i + \pi^* g_Z ,
\end{eqnarray*}
while squashing the fibers of $X \rightarrow Z$ by a factor of $\sqrt{5}$ we obtain the metric
\begin{eqnarray*}
	g^{np}=\frac{1}{5}\sum\eta^i\otimes \eta^i +\pi^*g_Z
\end{eqnarray*}
This is obtained from the $\G_2$-structure determined by the $3$-form 
\begin{eqnarray*}
	\varphi^{np} =5^{-3/2}\eta^{123}-5^{-1/2}(\eta^1\wedge \omega^1+\eta^2\wedge \omega^2+\eta^3\wedge \omega^3),
\end{eqnarray*}
is also Einstein, and the metric cone over $g^{np}$ has holonomy equal to $\Spin(7)$. Such metrics are known as nearly parallel and the above construction was first discovered in \cites{FKMS1997}. We refer to the survey article \cites{Boyer2001} and references therein for more information on $3$-Sasaki geometry.

\subsection*{Evolution equations}

Considering $X$ as a principal $\SU(2)$-bundle over $Z$, we can construct the associated (orbi-)vector bundle $E=X\times_{\SU(2)} \mathbb{C}^2$ over $Z$. Away from the singular locus of $Z$, $E$ is a smooth vector bundle, with $Z$ smoothly embedded in $E$ as the zero section. Moreover, $E\setminus Z$ is diffeomorphic to the product $X\times (0,\infty)$. Typically in this work, we will suppose that $E$ is a smooth $8$-manifold, although how the $\Spin(7)$ structures and $\Spin(7)$ instantons degenerate on the singular strata could also lead to interesting questions.\\
For $t, \nu \in \mathbb{R}^+$, we consider $X$ equipped with the $\G_2$-structure $\varphi$
$$\varphi = t^3 \eta_1 \wedge \eta_2 \wedge \eta_3 - t \nu^2 \frac{s}{48} \left( \eta_1 \wedge \omega_1 + \eta_2 \wedge \omega_2 + \eta_3 \wedge \omega_3 \right),$$
which determines the Riemannian metric $g_{\varphi}=t^2 \eta_i \otimes \eta_i + \frac{s\nu^2}{48} \pi^* g_Z$, the orientation $-\omega_1^2 \wedge \eta_{123}$ and associated $4$-form
$$\psi = \frac{\nu^4}{6} \left(\frac{s}{48}\right)^2 \omega_i \wedge \omega_i - \nu^2 t^2 \frac{s}{48} \left( \eta_1 \wedge \eta_2 \wedge \omega_3 + \eta_2 \wedge \eta_3 \wedge \omega_1 + \eta_3 \wedge \eta_1 \wedge \omega_2 \right).$$
Now we let $t, \nu$ be functions of $r$ and consider the $\Spin(7)$-structure
\begin{equation}\label{eq:Spin(7)_Structure}
\Theta=dr \wedge \varphi + \psi ,
\end{equation}
on $\mathbb{R}^+ \times X$. We then have
\begin{eqnarray*}
g_\Theta=dr^2+t(r)^2\sum_i\eta_i\otimes \eta_i +\nu(r)^2\pi^*g_Z,
\end{eqnarray*}
and we note that, for any $x\in X$, the curve $r\mapsto (r,x)\in (0,\infty)\times X$ is a geodesic for this metric.
The metric has holonomy in $\Spin(7)$ if and only if $\Theta$ is closed, i.e.
$$\frac{\partial \psi}{\partial r}=  d\varphi.$$
This on the other hand turns into the ordinary differential equations
\begin{eqnarray}\label{eq:ode_nu}
\frac{\partial \nu^2}{\partial r} & = &  6 t \\ \label{eq:ode_t}
\frac{\partial t}{\partial r} & = & 1-\frac{2 t^2}{\nu^2}.
\end{eqnarray}

\subsection*{The solutions} We start with the following example corresponding to the conical $\Spin(7)$-structure that exists on the cone over any nearly-parallel $\G_2$-structure.

\begin{example}\label{ex:cone}
The cone over the strictly nearly parallel structure obtained by squashing the $3$-Sasakian one is obtained by setting $t^2=c_1^2 r^2$ and $\nu^2=c_2^2 r^2$ for some constants $c_1 , c_2$. Inserting this into the equations we get that $c_1=3/5$ and $c_2 = 3/ \sqrt{5}$.\\
For this solution we can write $\varphi=r^3 \varphi_{np}$, and $\psi=r^4 \psi_{np}$ with the $\G_2$-structure $\varphi_{np}$ satisfying 
$$d \varphi_{np} = -4 \psi_{np}.$$
This is the (proper) nearly parallel structure of \cites{FKMS1997}, whose associated metric $g_{\varphi_{np}}$ is (up to scaling) obtained by squashing the $\SU(2)$-fibers of the $3$-Saskian one $g_7$.
\end{example}


\noindent The solutions to the system above can be written as the integral curves of the $1$-form
\begin{equation}\label{eq:1_Form_Omega} 
\omega = (\nu^2-2t^2)d \nu - 3 t \nu dt ,
\end{equation}
in $\mathbb{R}^2_+$ with coordinates $(\nu,t)$. The $1$-form above is homogeneous and so can easily been integrated, so that its integral curves are given implicitly as level sets of $\nu^4(\nu^2-5t^2)^3$, a few of which are plotted in Figure \ref{fig:Countorplot}. If we take initial conditions $t(0)=0$, $\nu(0)=\nu_0>0$, then the curve $(t,\nu)$ is constrained to satisfy $\nu^4(\nu^2-5t^2)^3=\nu_0^{10}>0$. Therefore, $1-\frac{2t^2}{\nu^2}\geq 1-\frac{5t^2}{\nu^2}>0$ so $t$ and $\nu$ are both strictly increasing for $r>0$.

\begin{figure}[]
\centering
	\includegraphics[scale=0.4]{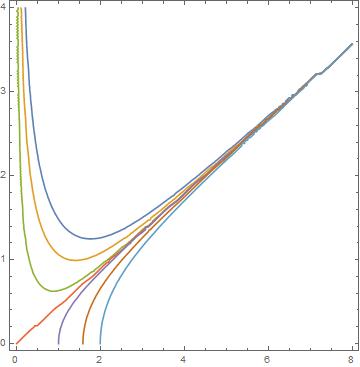}
\caption{\label{fig:Countorplot} Level sets of $\nu^4(\nu^2-5t^2)^3$, with $\nu$ the horizontal and $t$ the vertical axis.}
\end{figure}

\noindent It is immediate from the above that any solution curve asymptotes as $r \rightarrow + \infty$ to the line $\nu^2=5t^2$. In particular, this shows that the resulting structure is asymptotic, as $r \rightarrow + \infty$, to the conical one described in Example \ref{ex:cone}. It is also clear that those solution curves which start, when $r=0$ say, at $t=0$ actually compactify on a $\mathbb{R}^4$ (orbi)-bundle $E$ over $Z^4$, whose sphere bundle is $X^7$. Indeed, in this case, for $0\leq r \ll 1$ we have
\begin{eqnarray}\label{eq:Metric_Taylor_Expansion_nu}
\nu(r) & = & \nu_0 +\frac{3}{2 \nu_0} r^2 - \frac{13}{8 \nu_0^2} r^4 + \ldots \\ \label{eq:Metric_Taylor_Expansion_t}
t(r) & = & r - \frac{2}{3 \nu_0^2} r^3 + \ldots ,
\end{eqnarray}

where $\nu(0)^4=\nu_0^4 \in \mathbb{R}^+$ parametrizes the size of the zero section $Z \subset E$. Hence, the metric
\begin{eqnarray*}
g_{\Theta}= dr^2 + t^2 \sum_{i=1}^3 \eta_i \otimes \eta_i + \frac{s\nu^2}{48} \pi^* g_Z
\end{eqnarray*}
on $\mathbb{R}^+ \times X \cong E \backslash Z$, considered as an $\mathbb{R}^4\setminus\{0\}$-bundle over $Z$,   extends smoothly over the zero section. By the calculations that appear in the appendix, considering $\nu$ as an independent variable, the metric takes the form
\begin{eqnarray*}
g_{\Theta}=\frac{5}{9}\frac{1}{1-\left(\nu_0\nu^{-1}\right)^{10/3}}d\nu^2 +\frac{1}{5}\frac{1}{1-\left(\nu_0\nu^{-1}\right)^{10/3}}\nu^2\sum_i\eta_i\otimes \eta_i +\frac{s\nu^2}{48} \pi^* g_Z.
\end{eqnarray*}
The remaining solutions are either the conical one, in Example \ref{ex:cone}, and other incomplete $\Spin(7)$-metrics on the complement of the zero section in $Z$. Consider the connection on $E$ induced by the Levi-Civita connection on the anti-self-dual, Einstein $Z$. These incomplete $\Spin(7)$-metrics have the horizontal directions collapsing as one approaches the zero section while the normal spheres are getting infinitely large in all directions. 

\section{$\Spin(7)$-instantons}

\subsection*{Evolution equations}

We shall write the connection on the complement of the zero section $Z$, i.e. on $\mathbb{R}_+ \times X$, in radial gauge $A=a(r)$, where $r \in \mathbb{R}_+$. We can then interpret $A$ as an evolving family of connections on $X$. The curvature of $A$ on $\mathbb{R}^+_r \times X$ and its curvature is
$$F_A = F_a + dr \wedge \frac{\partial a}{\partial r},$$
where $F_a= F_a (r)$ denotes the curvature of $a(r)$ as a connection on a bundle over $\lbrace r \rbrace \times X$.\\
Taking $\ast$ of Equation \ref{eq:Spin7Inst_2}, the $\Spin(7)$-instanton equation may be written as
\begin{equation}\label{eq:Spin(7)_Instanton}
\ast \left( F_A \wedge \Theta \right) = - F_A ,
\end{equation}
and having in mind the $\Spin(7)$-structure \ref{eq:Spin(7)_Structure} we compute $\ast \left( F_A \wedge \Theta \right) =  \ast_{g_{\varphi}} ( F_a \wedge \varphi + \partial_r a \wedge \psi ) + dr \wedge \ast_{g_{\varphi}} \left( F_a \wedge \psi \right) $. Hence, the $\Spin(7)$ instanton equations turn into
\begin{equation}\label{eq:Spin(7)_Instanton_Evolution_Equation}
\frac{\partial a}{\partial r} = - \ast_{g_{\varphi}} \left( F_a \wedge \psi \right).
\end{equation}
We now turn to the connections $a$ that we are going to consider. To this point we have considered the trivial principal $\SU(2)$-bundle on $E\setminus Z$. For great ease of calculation, from this point on we consider the (trivial) associated vector bundle $V$, associated to the standard representation of $\SU(2)$ on $\mathbb{C}^2$. Given the projection $\pi:E\setminus Z\to Z$ described above, $V$ can be identified as $V\cong \pi^*E$. This bundle extends across the submanifold $Z\subseteq E$ in several ways. In the sequel (see Corollary \ref{cor:Extending_Smoothly}) we will be interested in the extensions 
\begin{enumerate}
\item $E\times \mathbb{C}^2\to E$, (trivial bundle on $E$),
\item $\pi^*E\to E$. 
\end{enumerate}
Initially we consider connections on the bundle $V\to E\setminus Z$. Later we will study how they extend across $Z$.

Recall that $X \rightarrow Z$ is the principal $\SU(2)$-bundle so that $\Lambda^2_+$ is the associated vector bundle with respect to the standard representation of $\SU(2)$ on $\mathbb{R}^3$. Moreover, as $E = X \times_{\SU(2)} \mathbb{C}^2$, we have that $E \backslash Z \cong \mathbb{R}^+ \times X$. Hence, using a radial gauge, it is enough to write the connections $a(r)$ as $1$-forms on $X$ with values on $\mathfrak{so}(3)$. So we write
\begin{equation}\label{eq:Connection_Form}
a(r)= \sum_{i=1}^3 a_i(r) T_i \otimes \eta_i,
\end{equation}
and using the fact that $d \eta_i = - \frac{s}{24} \omega_i - 2 \eta_j \wedge \eta_k$ its curvature can be computed to be
\begin{eqnarray}
F_a = \sum_{i=1}^3 \left( - \frac{s}{24} a_i \omega_i + \epsilon_{ijk} (a_ja_k - a_i) \eta_{jk} \right) \otimes T_i .
\end{eqnarray}
To compute the $\Spin(7)$ equation in its form as an evolution equation \ref{eq:Spin(7)_Instanton_Evolution_Equation} requires
\begin{eqnarray}\nonumber
\ast_{g_{\varphi}} (F_a \wedge \psi) & = &   \left( \frac{s }{48} \right)^2 \left( 2 \nu^2 t^2 a_1 + \nu^4 (a_2 a_3 - a_1) \right) T_1 \otimes \ast_{g_{\varphi}} (\omega_1^2 \wedge \eta_{23}) + \ldots \\ \nonumber
& = & \frac{2}{t} \left(  a_2 a_3 - \left( 1-\frac{2t^2}{\nu^2} \right) a_1 \right) T_1 \otimes\eta_1 + \ldots ,
\end{eqnarray}
where the dots stand for cyclic permutations of $(1,2,3)$ and we used $\dvol_{\varphi}= \frac{1}{2} \left( \frac{s}{48} \right)^2 t^3 \nu^4 \eta^{123} \wedge \omega_1^2$. Hence $A= a(r)$ is a $\Spin(7)$-instanton if and only if
\begin{eqnarray}\nonumber
\frac{\partial a_1}{\partial r} = - \frac{2}{t}  \left(  a_2 a_3 - \left( 1-\frac{2t^2}{\nu^2} \right) a_1 \right),
\end{eqnarray}
together with the similar looking equations obtained from cyclic permutation of $(1,2,3)$. In particular, in the case when $a_1=a_2=a_3=:x$ the connections \ref{eq:Connection_Form} become
\begin{equation}\label{eq:Connection_Form_Symmetric}
a(r)= \sum_{i=1}^3 x(r) T_i \otimes \eta_i, \ \text{with curvature} \ F_a = \sum_{i=1}^3 x \left( \epsilon_{ijk} (x - 1) \eta_{jk} - \frac{s}{24} \omega_i  \right) \otimes T_i,
\end{equation}
and the equation for $A=a(r)$ to be a $\Spin(7)$-instanton turns into
\begin{eqnarray}\label{eq:ODE}
\frac{\partial x}{\partial r} = - \frac{2}{t}  x \left(  x - \left( 1 - \frac{2t^2}{\nu^2} \right)  \right).
\end{eqnarray}

\begin{example}\label{ex:cone_Connection}
	In the case of the conical metric on $E \backslash Z$, i.e. when $t=3r/5$ and $\nu = 3r / \sqrt{5}$, we have $t^2 / \nu^2 = 1/5$ so it is easy to see that $x=0$ or $x=3/5$ is a solution. The connection with $a=0$ is the trivial flat connection and the connection with $x=3/5$ is the radial extension of the connection
	$$a_{\infty}= \frac{3}{5} \sum_{i=1}^3 \eta_i \otimes T_i , $$
	on the trivial $\mathbb{C}^2$-bundle over $X$. In fact, it is easy to see, using the static version of the evolution equation \ref{eq:Spin(7)_Instanton_Evolution_Equation}, that $a_{\infty}$ is actually a $G_2$-instanton for the nearly parallel $G_2$-structure of \cites{FKMS1997}{}, existing in any squashed $3$-Sasakian. For future reference we shall refer to $a_{\infty}$ as the \emph{canonical $\G_2$-instanton} of the nearly parallel $X$.
\end{example}

The rest of this paper is devoted to construct all solutions to the equations \ref{eq:ODE} which give rise to smooth $\Spin(7)$-instantons for the complete $\Spin(7)$-holonomy metric on $E$ obtained from $\nu(0)=\nu_0>0$.

\section{Solutions from ODE analysis}

In this section we prove a classification result for $\Spin(7)$-instantons for $\Theta$ on $E$ which arise from the evolution equation \ref{eq:ODE} from before. We start by settling the following.

\begin{lemma}\label{lem:Bounds}
\begin{itemize}
\item If $x > \left( 1-\frac{2t^2}{\nu^2} \right)$, then $x$ is decreasing.
\item If $ \left( 1-\frac{2t^2}{\nu^2} \right) >  x >  0$, then $x$ is increasing.
\item If $x(r_0)=0$, for some $r_0 >0$, then $x(r)=0$ for all $r \geq 0$. In particular if $x$ is positive somewhere, then it always remain positive.
\end{itemize}
\end{lemma}
\begin{proof}
Notice that for any of the complete $\Spin(7)$-structures we are considering $\nu^2-5t^2>0$ (see Figure \ref{fig:Countorplot}). Indeed, this quantity is positive at $0$ and cannot change sign as $\nu^2 = 5t^2$ is an integral curve of the $1$-form $\omega$ in Equation \ref{eq:1_Form_Omega}, and its integral curves cannot intersect in the positive quadrant. Hence, we also have that $1-2t^2/\nu^2>0$. Then, the claims in the 3 bullets follow easily from analysing the signs appearing in the right hand side of Equation \ref{eq:ODE}, and the standard existence and uniqueness theorem for ordinary differential equations.
\end{proof}

\begin{corollary}
If $x(r_0)>0$ for some $r_0\in(0,\infty)$, then the solution $x(r)$ is bounded and exists for all $r\in(0,\infty)$.
\end{corollary}

\begin{proof}
For any other $r_1\in(0,\infty)$, on the compact interval $K=[r_0,r_1]$, by the previous lemma the values that $x(r)$ can take are necessarily bounded between
\begin{eqnarray*}
0<\min\left\{\min_K\left(1-\frac{2t^2}{\nu^2}\right),x(r_0)\right\} \leq x(r)\leq \max\left\{\max_K \left(1-\frac{2t^2}{\nu^2} \right),x(r_0)\right\}.
\end{eqnarray*}
Therefore, by a comparison with the equation $\frac{dx}{dt}=Mx$ for $M$ sufficiently large, we can conclude that the maximal interval of the solution $x$ includes $K$.
\end{proof}
We now wish to understand the behaviour of the connections for small values of $r$, and in particular which solutions extend smoothly over the zero section $Z=\{r=0\}$ in $E$. Secondly, we wish to determine their limiting behaviour along the conical end.  
 For the first of these points, we benefit from having the following criterion, which may
  be of independent interest given its practicality in the cohomogeneity-$1$ setting.

\begin{proposition}\label{prop:Extending_Smoothly}
Let $A$ be a $\Spin(7)$-instanton for the structure $\Theta$ on $E \backslash Z$ as above. Then, $A$ smoothly extends over $Z$ up to gauge if and only if its curvature remains bounded.
\end{proposition}
\begin{proof}
From the gauge invariance of the norm of the curvature it follows that for a connection to smoothly extend over the zero section its curvature must remain bounded. Hence, this is certainly a necessary condition. For the converse we use the instanton condition and appeal to the removable singularity theorem of Tao and Tian \cites{Tao2004}. Indeed, if the curvature is bounded and $B_r(x)$ denotes a radius $r$ ball centred at a point $x \in Z$, then we certainly have that
$$\limsup_{r \rightarrow 0} r^{4-n} \int_{B_r(x)} |F_A|^2 =0,$$
where we note that here $n=8$. Thus, Tao-Tian's removable singularity theorem applies and $A$ smoothly extends over $Z$.
\end{proof}

We shall now use this to prove the following result.

\begin{corollary}\label{cor:Extending_Smoothly}
	Let $A$ be a $\Spin(7)$-instanton for $\Theta$ on $E \backslash Z$ which in radial gauge can be written as in Equation (\ref{eq:Connection_Form_Symmetric}). Suppose $A$ smoothly extends over $Z$, then either:
	\begin{itemize}
		\item $A|_Z$ is the connection on $E$ induced by the Levi-Civita connection of the anti-self-dual, Einstein metric on $Z$, in which case:
		$$x(r)=1 + O(r^2).$$
		\item $A|_Z$ is the flat connection on the trivial $\mathbb{C}^2$-bundle over $Z$, in which case:
		$$x(r)=O(r^2)$$
	\end{itemize}
\end{corollary}
\begin{proof}
Appealing to Proposition \ref{prop:Extending_Smoothly} we shall simply impose the condition that the curvature remains bounded. For this, and using the inner product $\langle A , B \rangle =-\tr(AB)$ on $\mathfrak{su}(2)$, we compute
\begin{eqnarray}\nonumber
\frac{1}{2} | F_A |^2 & = & \frac{1}{2} \left( | \dot{a} |^2 + | F_a |^2 \right) \\ \label{eq:Norm_curvature}
& = &3 \frac{( \partial_r x )^2}{t^2} + 12 \frac{x^2}{\nu^4} + 12 \frac{x^2(x-1)^2}{t^4} ,
\end{eqnarray}
and recalling that for $r \ll 1$ we have $t(r) = r + O(r^3)$ while $\nu(r) = \nu_0+O(r^2)$, we conclude that either:
\begin{itemize}
\item $x(r)=1 + O(r^2)$, or
\item $x(r)=O(r^2)$.
\end{itemize}
These initial conditions can be geometrically interpreted as follows. The connections $A$ considered above live on the bundle $\pi^* E$, where $\pi: E \backslash Z \rightarrow Z$. This is trivial as an $\mathbb{C}^2$-bundle and extends as a bundle over the whole of $E$ by gluing it with either $E$ or $Z \times \mathbb{C}^2$ over the zero section $Z$. In particular, the connection 
$$a(0):= \lim_{r \rightarrow 0} a(r)$$ 
induces a connection either on $E\to Z$ or $Z \times \mathbb{C}^2\to Z$. For connections satisfying the first case above, we have 
$$a(0)= \sum_{i=1}^3 \eta_i \otimes T_i,$$ 
which is the Levi-Civita connection on $X$ associated with the anti-self-dual Einstein metric in $Z$. This defines a connection, on $E\to Z$, so $A$ extends across $Z$ to give a connection on $\pi^*E\to E$. On the other hand, for the second case above we have $a(0)=0$ and so it is a connection on the trivial bundle $Z \times \mathbb{C}^2$ and $A$ extends to give a connection on $E\times \mathbb{C}^2$. In summary, the initial conditions above determine whether $A$ extends over the zero section as a connection on $E$, or on the trivial bundle over $Z$. 
\end{proof}
 
We now analyze the local existence around the zero section of solutions satisfying either of these (singular) initial value problems.

\begin{theorem}\label{thm:Solutions_ODE_Analaysis}
Let $\Theta$ be the torsion free $\Spin(7)$-structure from Section \ref{sec:Structure} and $A$ a $\Spin(7)$-instanton for $\Theta$ which in radial gauge may be written as \ref{eq:Connection_Form_Symmetric}. Then $A$ is asymptotic as $r \rightarrow + \infty$ to the canonical $\G_2$-instanton $a_{\infty}$ of the nearly parallel $X^7$, and either:
\begin{itemize}
	\item $A|_Z$ is the connection on $E$ induced by the Levi-Civita connection of the anti-self-dual, Einstein metric on $Z$. In this case there is a unique $\Spin(7)$-instanton $A_{\lim}$ on $\pi^*E \rightarrow E$ with this property.
	\item $A|_Z$ is the flat connection on the trivial $\mathbb{C}^2$-bundle over $Z$, in which case there is a real $1$-parameter family of such $\Spin(7)$-instantons with this property. These live on the trivial bundle $E \times \mathbb{C}^2$, are parametrized by $y_0 \in [0, + \infty)$ and can be written as
	$$A_{y_0}=r^2 {y}(r) \sum_{i=1}^3 \eta_i \otimes T_i,$$
	for some real analytic function ${y}$ of $r$, with $\tilde{y}(0)=y_0 \geq 0$. For $y_0<0$ we obtain $\Spin(7)$-instantons which are only locally defined in a neighbourhood of $Z$.
\end{itemize}
\end{theorem}
\begin{proof}
We start by addressing the initial value problem with the case when $x(r)=1+O(r^2)$, in which case we shall write $x(r)=1+r^2 y(r)$ for some real analytic $y$. Then, the ODE \ref{eq:ODE} turns into 
$$\frac{\partial y}{\partial r} = - \frac{4}{r} \left( y+  \frac{1}{\nu_0^2} \right) +  f(r,y),$$
where $f(r,y)$ is real analytic in both entries. Then, by a theorem of Malgrange in \cites{Malgrange1974}, see also Theorem 4.7 in \cites{Foscolo2015}, a local solution to this equation exists if and only if $y(0)=- \frac{1}{\nu_0^2}$.

We now turn our attention to the case of $x(r)=O(r^2)$, for that we shall write $x(r)=r^2 y(r)$, for some real analytic $y$. Then, using the Taylor expansion for the metric as in Equations \ref{eq:Metric_Taylor_Expansion_nu}--\ref{eq:Metric_Taylor_Expansion_t}, the ODE \ref{eq:ODE} turns into 
$$\frac{\partial y}{\partial r} = -2y \left( y+ \frac{4}{3 \nu^2} \right) r + f(r,y),$$
where $f(r,y)$ is real analytic in both entries and of strictly higher order in $r$. Thus, the standard local existence and uniqueness theorem for ODE's applies and guarantees the existence of a unique solution parametrized by $y(0)$. For $y(0)=0$ this solution is clearly $y=0$, so $x=0$ as well and corresponds to the flat connection. For $y(0) >0 $ (resp. $y(0)<0$) the resulting solution $x(r)$ is positive (resp. negative) for all $r>0$ as solutions cannot cross zero for $r>0$ (by the standard existence and uniqueness theorem again).

We start by analysing the case when $y(0)<0$, in which case by having in mind that $1-2t^2/\nu^2 >0$ we have 
$$\frac{\partial x}{\partial r} < - \frac{2x^2}{t(r)}$$ 
which is a separable differential inequality. This can be integrated to give
$$x(r)^{-1} > x(\epsilon)^{-1} + 2 \int_{\epsilon}^r \frac{dr}{t(r)},$$
for any $r>\epsilon>0$. Recall (see for example Equations \ref{eq:t_In_Terms_Of_nu}  and \ref{eq:dr^2})   that for large $r$ we have $t(r)=O(r)$ and so $\int_{\epsilon}^r \frac{dr}{t(r)} = O(\log(r))$ which is unbounded. As $x(\epsilon)=\epsilon^2y(\epsilon)<0$, there is $r_0 < + \infty$ such that the right hand side vanishes and so $x$ must explode at a finite distance to the $r=0$. This proves that the resulting instantons are only locally defined.

The remaining case is when $y(0)>0$, in which case $x(r)>0$ for all $r>0$. However, $x(r)$ leads off as $O(r^2)$ and so by falls into the case of the second bullet in Lemma \ref{lem:Bounds} and so is increasing satisfying
$$ 0 <  x <  1-\frac{2t^2}{\nu^2} .$$
Hence, the limit of $x_{\infty}= \lim_{r \rightarrow + \infty} x(r)$ exists and is finite. In fact, as we shall now show, we must have
$$x_{\infty}=  \lim_{r \rightarrow + \infty} \left( 1- \frac{2t^2}{\nu^2} \right)= \frac{3}{5}.$$ 
To show this, notice that since $x(r)$ has a finite positive limit when $r \rightarrow \infty$, we must have $\partial_r x$ converging to zero in that limit. Hence, returning to equation \ref{eq:ODE}, which in terms of $u(r)=\int_\varepsilon^r t(s)^{-1} ds$ becomes 
$$\frac{\partial x}{\partial u} = - 2x \left( x- \left( 1-\frac{2t^2}{\nu^2} \right) \right),$$
we must have that the right hand side goes to zero when $r \rightarrow + \infty$, i.e. $u (r) \sim \log (r^{3/5}) \rightarrow + \infty$. Again as $x$ is increasing it cannot go to zero, so that we must have
$$\lim_{r \rightarrow + \infty} x(r) = 1- \lim_{r \rightarrow + \infty} \frac{2t^2}{\nu^2} = \frac{3}{5}. $$
We shall now return to the solution with $x(0)=1$. Then, having in mind that $x(r)=1-\frac{r^2}{\nu_0^2}>1-\frac{2r^2}{\nu_0^2}$ is now in the setup of the first bullet of Lemma \ref{lem:Bounds} and doing exactly the same analysis also yields that this instanton is asymptotic to the canonical connection $a_{\infty}$. 
\end{proof}

\begin{remark}
	As alluded to during the proof of Corollary \ref{cor:Extending_Smoothly}, the instantons constructed in the previous Theorem live either on $\pi^*E$ or the trivial one $E \times \mathbb{C}^2$. 
\end{remark}

\section{Explicit formulas for the solutions}

In order to find explicit solutions to Equation \ref{eq:ODE}, and inspired by the form of those in the second bullet of Theorem \ref{thm:Solutions_ODE_Analaysis} we shall transform the Equation \ref{eq:ODE} as an equation for the function $y(r)=\frac{x(r)}{t(r)^2}$. Indeed, in terms of $y(r)$, the connection can be written as
$$A= t(r)^2 y(r) \sum_{i=1}^3 \eta_i \otimes T_i,$$
and given that $t(r) = r + \ldots$, comparing with the second bullet of Theorem \ref{thm:Solutions_ODE_Analaysis} we have $y(0)=\tilde{y}(0)=y_0$. Thus, the solutions in that second bullet are parametrized by $y(0)$. Regarding the Equation \ref{eq:ODE}, a simple computation shows that in terms of $y(r)$ it becomes
\begin{equation}\label{eq:ODE_2}
\frac{\partial y}{\partial r} = -2t(r) y^2(r).
\end{equation}
This can be integrated using the initial condition $y(0)=y_0$ to give
\begin{equation}
y_{y_0}(r)=\frac{y_0}{1 + 2 y_0 \int_0^r t(r) \ dr} .
\end{equation}
Indeed, in agreement with Theorem \ref{thm:Solutions_ODE_Analaysis} we can check that these exist in all of $E$ for $y_0\geq 0$ while for $y_0<0$ they explode at finite distance from $Z$, i.e. they are only defined for small $r$. Furthermore, when $y_0=0$ the solution corresponds to the flat connection. As in \cite{Lotay2018} we shall now consider the moduli space of these solutions, which we may parametrize by the coordinate $y_0 \in [0,\infty)$. As in that reference we expect that in the limit when $y_0 \rightarrow +\infty$, the instantons $A_{y_0}$ geometrically converge on the complement of $Z$ to a connection $A_{\lim}$ from the first bullet of Theorem \ref{thm:Solutions_ODE_Analaysis}.

Indeed, we can also write $A_{\lim}$ explicitly by naively taking the limit $y_0 \rightarrow +\infty$ for $r \neq 0$. This gives $y_{\lim}(t) = \frac{1}{2 \int_0^r t(r) dr}$ which corresponds to 
\begin{eqnarray*}
x_{\lim}= \frac{t(r)^2}{2 \int_0^r t(r) \ dr}.
\end{eqnarray*}
With this formula we can verify that $x_{\lim}$ satisfies Equation \ref{eq:ODE} and is of the form $x_{\lim}(r)=1+O(r^2)$.
In summary we have shown that

\begin{proposition}\label{prop:Explicit_Formulas}
	The $\Spin(7)$-instantons from Theorem \ref{thm:Solutions_ODE_Analaysis} can be explicitly written as
	$$A_{y_0} = \frac{y_0 \ t^2(r)}{1 + 2 y_0 \int_0^r t(s) \ ds} \sum_{i=1}^3 \eta_i \otimes T_i , \ \text{and} \ \ \ A_{\lim}=\frac{t^2(r)}{2 \int_0^r t(r) \ dr} \sum_{i=1}^3 \eta_i \otimes T_i . $$
\end{proposition}

\begin{remark}
	 We can rewrite the explicit formulas above in terms of $\nu^2$ instead of $r$. This is done using the formulas \ref{eq:t_In_Terms_Of_nu} and \ref{eq:nu_In_Terms_Of_Int_t} for $t^2(r)$ and $\int_0^r t(s) ds$ respectively, which we derive in the Appendix. As a result of these and a short computation we obtain
	 $$A_{y_0} = \frac{3}{5} \left( 1+ \frac{\nu_0^2-3/y_0}{\nu^2-\nu_0^2 + 3/y_0} - \left(\frac{\nu_0}{\nu}\right)^{\frac{4}{3}} \frac{\nu_0^2}{\nu^2-\nu_0^2 + 3/y_0} \right) \sum_{i=1}^3 \eta_i \otimes T_i ,$$
	 and also
	 $$A_{\lim} = \frac{3}{5} \left( 1+ \frac{\nu_0^2}{\nu^2-\nu_0^2 } - \left(\frac{\nu_0}{\nu}\right)^{\frac{4}{3}} \frac{\nu_0^2}{\nu^2-\nu_0^2 } \right) \sum_{i=1}^3 \eta_i \otimes T_i .$$
\end{remark}

Recall from Theorem \ref{thm:Solutions_ODE_Analaysis} that both $A_{\lim}$ and the $A_{y_0}$ are asymptotic to the canonical $\G_2$-instanton $a_{\infty}$ on the nearly parallel $G_2$-manifold $X$. However, equipped with the formulas in Proposition \ref{prop:Explicit_Formulas} it is easy to refine our study of the asymptotic behaviour of these instantons. In particular, we shall be able to show some new phenomena which did not appear in the analogous case of \cites{Lotay2018} for $\G_2$-instantons. In that case the analogue of $A_{\lim}$ was asymptotic to the relevant limiting connection at a strictly faster rate than $A_{y_0}$. For us this will not be the case and we will show that $A_{\lim}$ is asymptotic to $a_{\infty}$ at the same rate as the connection $A_{y_0}$ does (for generic $y_0$). However, there is one specific value $y_{0}=3/\nu_0^2$ for which $A_{3/\nu_0^2}$ decays at a strictly faster rate. The precise statement is the following

\begin{proposition}\label{prop:Asymptotics}
	Let $A$ be one of the (non-flat) $\Spin(7)$-instantons constructed in Theorem \ref{thm:Solutions_ODE_Analaysis} and defined in the whole total space of $E$. Then $A$ is asymptotic to the canonical $\G_2$-instanton $a_{\infty}$ defined in Example \eqref{ex:cone_Connection} on the nearly parallel asymptotic cross section $X^7$. Then, 
	\begin{itemize}
		\item If $A$ lives on $\pi^*E$, i.e. if $A=A_{\lim}$ then for sufficiently large $r$
		$$| A_{\lim} - a_{\infty} | \leq C\nu_0^2 r^{-3} ,$$
		for some positive constant $C>0$ independent of $y_0$;
		\item If $A$ lives on the trivial $\mathbb{C}^2$-bundle, i.e. $A=A_{y_0}$ for some $y_0 > 0$, then
		$$|A_{y_0} - a_{\infty}| \leq C | \nu_0^2  -3/y_0| r^{-3} .$$
	\end{itemize}
\end{proposition}
\begin{proof}
We start with the computation for $A_{y_0}$, for which we may use the explicit formulas above and recall that for $t^2(r)$ and $\int_0^r t(s) ds$ we have explicit formulas in terms of $\nu$, noting that for large $r$, $\nu=O(r)$. We also note here that the estimates are calculated with respect to the metric $g_\Theta$ on $1$-forms. These are given in the Appendix, namely in equations \ref{eq:t_In_Terms_Of_nu} and \ref{eq:nu_In_Terms_Of_Int_t}. Then, for large $r$ and so large $\nu$ we may compute
\begin{eqnarray*}
A_{y_0} & = & \frac{y_0 \ t^2(r)}{1 + 2 y_0 \int_0^r t(s) \ ds} \sum_{i=1}^3 \eta_i \otimes T_i = \frac{3}{5} \frac{1}{\nu^{4/3}} \frac{\nu^{10/3}-\nu_0^{10/3}}{\nu^{2}-\nu_0^{2} +3/y_0} \sum_{i=1}^3 \eta_i \otimes T_i \\
& = & \frac{3}{5} \left( 1 + \frac{\nu_0^2 y_0 -3}{y_0\nu^2} + \ldots \right) \sum_{i=1}^3 \eta_i \otimes T_i = a_{\infty} + \frac{3}{5}  \frac{\nu_0^2 y_0 -3}{y_0\nu^2} \sum_{i=1}^3 \eta_i \otimes T_i + \ldots, 
\end{eqnarray*}
where the $\ldots$ stand for lower order terms. As a consequence of these and using the fact that $|\eta_i| \sim t^{-1} \sim \nu^{-1}$ we compute that there is a constant $C>0$ such that
\begin{equation}\nonumber
|A_{y_0} - a_{\infty}| \leq C \left| \frac{\nu_0^2 y_0 -3}{y_0}\right|\nu^{-3}.
\end{equation}
Proceeding in an analogous way we compute
\begin{eqnarray*}
	A_{\lim} & = & \frac{t^2(r)}{ 2 \int_0^r t(s) \ ds} \sum_{i=1}^3 \eta_i \otimes T_i = \frac{3}{5} \frac{1}{\nu^{4/3}} \frac{\nu^{10/3}-\nu_0^{10/3}}{\nu^{2}-\nu_0^{2} } \sum_{i=1}^3 \eta_i \otimes T_i \\
	& = & \frac{3}{5} \left( 1 + \frac{\nu_0^2 }{\nu^2} + \ldots \right) \sum_{i=1}^3 \eta_i \otimes T_i = a_{\infty} + \frac{3}{5}  \frac{\nu_0^2}{\nu^2} \sum_{i=1}^3 \eta_i \otimes T_i + \ldots, 
\end{eqnarray*}
and so $| A_{\lim} - a_{\infty} | \leq C\nu_0^2/\nu^3 $. To get this as stated recall that $\nu(r) = O(r)$ and so by possibly changing the constant $C>0$ we obtain the stated result.
\end{proof}

\begin{remark}
	The reader may have noticed that there is a special situation occurring in the previous theorem. This is the case when $y_0=3/\nu_0^2$ for which the connection is
	$$A=\frac{3}{5} \left( 1-  \left( \frac{\nu_0}{\nu} \right)^{10/3} \right) \sum_{i=1}^3 \eta_i \otimes T_i,$$
	and so is asymptotic to $a_{\infty}$ at a higher rate of $O(r^{-13/3})$. 
\end{remark}

\section{Moduli space compactness}

As done in \cite{Lotay2018} for the analogous case of $\G_2$-instantons, we shall show that as $y_0 \rightarrow + \infty$ the $\Spin(7)$-instantons $A_{y_0}$ geometrically converge to $A_{\lim}$ away from the zero section $Z$. As we shall also show, this is a Cayley submanifold along which the energy density concentrates and an anti-self-dual (asd) connection ``bubbles off'' along the transverse directions. One may interpret the fact that $A_{\lim}$ smoothly extends across $Z$ as an example of a removable singularity phenomenon.

\subsection{Bubbling and removable singularities}

We start with some preparation which is convenient for rigorously stating the bubbling phenomenon. Let $z \in Z$ and $\lambda >0$, then we define the scaled exponential map $s_{z,\lambda}$, with domain the unit ball $B_1\subseteq N_zZ\cong \mathbb{R}^4$ in the normal bundle to $Z$, as
\begin{eqnarray*}
s_{z,\lambda} : B_1\subseteq N_zZ \rightarrow B_{\lambda}(z)
  \subseteq E, \ \ x \mapsto \exp_z(\lambda x).
\end{eqnarray*}

Now we recall the basic ASD instanton on $\mathbb{R}^4$ with scale $\kappa >0$. In polar coordinates on $\mathbb{R}^4 \backslash {0}=\mathbb{R}^+_r \times S^3$, this is given by
\begin{equation}\label{eq:Basic_Instanton}
A^{asd}_{\kappa}= \frac{\kappa r^2}{1+\kappa r^2} \sum_{i=1}^3T_i\otimes\eta_i ,
\end{equation}
where these $\eta_i$ are a standard left-invariant coframing of $\SU(2)$.

\begin{theorem}\label{thm:Compactness}
	Let $\lbrace A_{y_0} \rbrace_{y_0}$ with $y_0 \nearrow + \infty$ be a sequence of $\Spin(7)$-instantons constructed in Theorem \ref{thm:Solutions_ODE_Analaysis}. Then, the following are true:
	\begin{itemize}
		\item[(a)] Given any $\kappa>0$, there is a nullsequence $\lambda=\lambda(y_0,\kappa)$ with the following significance: for all $z \in Z$, $(s_{z,\lambda})^* A_{y_0}$ uniformly converges (with all derivatives) to $A^{asd}_{\kappa}$ on $B_1\subseteq\mathbb{R}^4$ as in \eqref{eq:Basic_Instanton}. 
		\item[(b)] The connections $A_{y_0}$ uniformly converge (with all derivatives) to $A_{\lim}$ on all compact subsets of $E \backslash Z $. 
	\end{itemize}
\end{theorem}
\begin{proof}
	(a) We start with the proof of the first item. For this we use the explicit formula in Proposition \ref{prop:Explicit_Formulas} and the expansion of $t(r)$ in (\ref{eq:Metric_Taylor_Expansion_t}). Given these we compute 
	$$
	(s_{x,\lambda})^* A_{y_0}  =  \frac{y_0 \ t^2(\lambda r)}{1 + 2 y_0 \int_0^{\lambda r} t(s) \ ds} \sum_{i=1}^3 \eta_i \otimes T_i   =  \frac{ y_0 \lambda^2 r^2 + O(y_0 (\lambda r)^4) }{1+y_0 \lambda^2 r^2 + O(y_0 (\lambda r)^4)} \sum_{i=1}^3 \eta_i \otimes T_i .
	$$
	Choosing $\lambda =  \sqrt{\kappa / y_0}$ and letting $k \in \mathbb{N}_0$ we have
	\begin{eqnarray*}
		\Vert (s_{x,\lambda})^* A_{y_0} - A^{asd}_{\kappa} \Vert_{C^k(B_1)} \leq C_k y_0 \lambda^4 = \frac{\kappa^2}{y_0},
	\end{eqnarray*}
	for some positive constant $C_k$ not depending on $\kappa$ and $y_0$. As a consequence, for all $\epsilon >0$ there is $y_0 \geq C_k \kappa^2/ \epsilon$  such that
 	$$\Vert (s_{x,\lambda})^* A_{y_0}  - A^{asd}_{\kappa} \Vert_{C^k(B_1)} \leq \epsilon,$$
	which proves that indeed $ (s_{x,\lambda})^* A_{y_0}$ converges uniformly with all derivatives to $A^{asd}_{\kappa}$.\\
	(b) We now turn to the proof of the second claim. This is an easy consequence of the explicit formulas for $A_{y_0}$ and $A_{\lim}$ in Proposition \ref{prop:Explicit_Formulas}. Indeed, given these we have
	\begin{eqnarray}\nonumber
	\left| A_{y_0} - A_{\lim} \right| & = &  \left| \frac{y_0 \ t^2(r)}{1 + 2 y_0 \int_0^r t(s) \ ds} - \frac{ t^2(r)}{ 2 \int_0^r t(s) \ ds}   \right| \left| \sum_{i=1}^3T_i\otimes \eta_i \right| \\ \nonumber
	& \leq & C \frac{ t(r)  }{ 2 \int_0^r t(s) \ ds}  \left| \frac{ 2 y_0 \int_0^r t(s) \ ds }{1 + 2 y_0 \int_0^r t(s) \ ds} - 1  \right| \\
	& \leq & C \frac{ t(r)  }{ 2 \int_0^r t(s) \ ds}  \Big\vert \frac{1 }{1 + 2 y_0 \int_0^r t(s) \ ds}   \Big\vert
	\end{eqnarray}
	for some constant $C>0$. The first term above is independent of $y_0$ and on any compact set $K \subset E \backslash Z$ is uniformly bounded by some constant $C_K$ depending only on $K$ and not on $y_0$. As for the second term, this clearly converges to $0$ as $y_0 \nearrow + \infty$. In fact, by possibly changing the constant $C_K>0$, we have
	\begin{equation}
	\vert A_{y_0} - A_{\lim} \vert  \leq  \frac{C_K}{1+y_0},
	\end{equation}
	on $K \subset E \backslash Z$. Similar computations show that analogous estimates hold true for all the derivatives of $A_{y_0}-A_{\lim}$. Hence, as $y_0 \nearrow + \infty$ the right-hand side of these goes to zero and the result follows.
\end{proof}

\begin{remark}
	The fact that $A_{y_0} \rightarrow A_{\lim}$ away from the Cayley submanifold $Z$ but $A_{\lim}$ smoothly extends across $Z$ can be interpreted as an example of a removable singularity phenomenon. This is a special case of a more general result of Tao and Tian \cites{Tao2004} to which we have actually already appealed in Proposition \ref{prop:Extending_Smoothly}.
\end{remark}

\subsection{Energy conservation}

 We shall now follow the same strategy of \cites{Lotay2018} and prove an energy conservation formula for these instantons when we pass to the limit $y_0 \nearrow + \infty$. Indeed, as in that reference, the instantons we consider here have infinite energy, and so the results of \cites{Tian2000} cannot be naively applied. In order to make sense of an energy conservation formula we shall interpret the energy densities as currents and prove a weak convergence statement for these. Let $\delta_{Z}$ denote the current of integration associated with the zero section $Z \subset E$, which we recall is a Cayley submanifold. It is in fact the unique such compact suborbifold (or submanifold if $Z=\mathbb{S}^4$).

\begin{theorem}\label{thm:delta}
	Let $y_0>0$, then the function $| F_{A_{y_0}} |^2 - | F_{A_{\lim}} |^2$ is integrable. Moreover, as $y_0 \nearrow + \infty$ these functions converge as currents to $8 \pi^2 \delta_{Z}$. That is, for all compactly supported $f \in C^{\infty}_0(E, \mathbb{R})$ 
	$$\lim_{y_0 \rightarrow + \infty} \int_{E} f  (| F_{A_{y_0}} |^2 - \vert F_{A_{\lim}} |^2)  = 8 \pi^2 \int_{Z} f ,$$
	where in the latter integral we use the induced metric on the Cayley $Z \subset E$.
\end{theorem}
\begin{proof}
	First, using the fomula \eqref{eq:Norm_curvature} for the norm of the curvature we compute
	\begin{eqnarray}\label{eq:Curvature_Difference}
	| F_{A_{y_0}} |^2 - | F_{A_{\lim}} |^2 = \sum_{n=0}^3   \frac{ y_0^n (\nu^2-\nu_0^2)^n N_n(\nu,t) }{\nu^4 (\nu^2-\nu_0^2)^4 (y_0 (\nu^2-\nu_0^2) +3)^4} ,
	\end{eqnarray}
	for some completely explicit $N_n(\nu^2)$ of the form 
	\begin{eqnarray}\nonumber
	N_3(\nu^2) & = & (\nu^2-\nu_0^2)^4 f_3(\nu^2) \\ \nonumber
	N_2(\nu^2) & = & (\nu^2-\nu_0^2)^2 f_2(\nu^2) \\ \nonumber
	N_1(\nu^2) & = & (\nu^2-\nu_0^2)^4 f_1(\nu^2) \\ \nonumber
	N_0(\nu^2) & = & (\nu^2-\nu_0^2)^4 f_0(\nu^2) ,
	\end{eqnarray}
	with the $f_n(\nu^2)$ nonvanishing at $\nu^2=\nu_0^2$. For future reference we further mention that 
	\begin{equation}\label{eq:}
	f_2(\nu_0^2)=3888\nu_0^4.
	\end{equation}
	Now, recall that $6t d r= d\nu^2$ and so $\dvol_g = t^3 \nu^4 dr \wedge \dvol_{g^*}= \tfrac{1}{6} t^2 \nu^4 d\nu^2 \wedge \dvol_{g^*}$, where $g^*$ denotes the metric associated with the $\G_2$-structure $\varphi^* = \eta_{123}-\tfrac{s}{48} (\eta_1 \wedge \omega_1 + \eta_2 \wedge \omega_2 + \eta_3 \wedge \omega_3)$ on $X \cong r^{-1}(s), \ \forall_{s>0}$. Thus, as for large $r$ we have $t^2 =O(\nu^2)$, the integrability of $| F_{A_{y_0}} |^2 - | F_{A_{\lim}} |^2$ follows immediately from this formula and the fact that for all $n=0,1,2,3$ the $ (\nu^2-\nu_0^2)^nN_n=O((\nu^2)^5)$. Indeed, putting all this together we have 
	$$(| F_{A_{y_0}} |^2 - | F_{A_{\lim}} |^2) \dvol_g \sim (\nu^2)^{5+1-8} d \nu^2 \wedge \dvol_{g^*},$$
	and $ (\nu^2)^{-2} d \nu^2$ is integrable for large $\nu^2$.\\
	We shall now compute the current obtained as the limit of $| F_{A_{y_0}} |^2 - | F_{A_{\lim}} |^2$ as $y_0 \rightarrow + \infty$ and show it equals $8 \pi^2 \delta_Z$, the delta current of the Cayley $Z \subset E$. To achieve this it is enough to show that
	\begin{eqnarray}\label{eq:Limit}
	\lim_{y_0 \rightarrow + \infty}  \int_K (| F_{A_{y_0}} |^2 - | F_{A_{\lim}} |^2) \dvol_g 
	\end{eqnarray}
	vanishes for all compact $K \subset E \backslash Z$ and equals $8 \pi^2  \mathrm{Vol}(Z)$ for any $K$ containing $Z$. Implicit in this statement is the definition of $\delta_Z$ and the observation that each $| F_{A_{y_0}} |^2 - | F_{A_{\lim}} |^2$ is of the form $f_{y_0}(r)$ for some $f: \mathbb{R}_0^+ \rightarrow \mathbb{R}$ integrable with respect to the measure $t^3(r) \nu^4(r) \ dr$ on $\mathbb{R}^+$, and so is enough to consider subsets $K \subset E$ which can be written as sublevel sets of $r:E \rightarrow \mathbb{R}$.
	
	First, consider the case of a compact $K \subset E \backslash Z$. Recall, from part (b) of Theorem \ref{thm:Compactness}, that $| F_{A_{y_0}} |^2 - | F_{A_{\lim}} |^2$ converges uniformly to $0$ in $K$, and so from the dominated convergence theorem it follows that \eqref{eq:Limit} vanishes.
	
	In the case where $Z \subset K$ we have
	\begin{eqnarray}\nonumber
	\lim_{y_0 \rightarrow + \infty} & & \int_K (| F_{A_{y_0}} |^2 - | F_{A_{\lim}} |^2) \dvol_g = \\ \nonumber
	& &  = \lim_{y_0 \rightarrow + \infty} \int_0^{+ \infty} ds \ \int_{r^{-1}(s) \cap K} (| F_{A_{y_0}} |^2 - | F_{A_{\lim}} |^2) t^3(s) \nu^4(s) \dvol_{g^*}  \\ \nonumber
	& &  = \frac{1}{6} \lim_{y_0 \rightarrow + \infty} \int_{\nu_0^2}^{+ \infty} d\nu^2 \ \int_{r^{-1}(s) \cap K} (| F_{A_{y_0}} |^2 - | F_{A_{\lim}} |^2) t^2 \nu^4 \dvol_{g^*} 
	\end{eqnarray}
	and analyse what happens independently to each of the terms appearing in the formula \ref{eq:Curvature_Difference}. Denote by $\mathrm{Vol^*}$ the volume of $r^{-1}(1)$ with respect to $g^*$. Then, for $n \neq 2$, each of these takes the form
	\begin{eqnarray}\nonumber
	I_{n} & := &  \lim_{y_0 \rightarrow + \infty} \frac{y_0^{n-4}}{6} \int_{\nu_0^2}^{\nu_0^2 + \epsilon}  \frac{ (\nu^2-\nu_0^2)^n N_n(\nu,t)t^2 }{ (\nu^2-\nu_0^2)^4 ( \nu^2-\nu_0^2 +3/y_0 )^4}  \ d\nu^2 \ \mathrm{Vol^*} \\ \nonumber
	& = &  \lim_{y_0 \rightarrow + \infty} \frac{y_0^{n-4}}{18} \int_{\nu_0^2}^{\nu_0^2 + \epsilon}  \frac{ (\nu^2-\nu_0^2)^{n+1} f_n(\nu) }{ ( \nu^2-\nu_0^2 +3/y_0 )^4}  \ d\nu^2 \ \mathrm{Vol^*} + \ldots \\ \nonumber
	& = &  \lim_{y_0 \rightarrow + \infty}  \frac{y_0^{n-4}}{18}  \int_{0}^{\epsilon}  \frac{ x^{n+1} f_n(\nu_0) }{  ( x +3/y_0)^4}  \ dx \ \mathrm{Vol^*} + \ldots\\ \nonumber
	& = & \lim_{y_0 \rightarrow + \infty} O(y_0^{-2}) \\ \nonumber
	& = & 0 ,
	\end{eqnarray}
	where in the first equality we used the fact that for small $r$ we have $t^2= \frac{1}{3}(\nu^2-\nu_0^2)-\frac{1}{9}(\nu^2-\nu_0^2)^2 + \ldots$ with the $\ldots$ denoting higher order terms.\\
	As a consequence of this computation we have
	\begin{equation}
	\lim_{y_0 \rightarrow + \infty} | F_{A_{y_0}} |^2 - | F_{A_{\lim}} |^2  = \lim_{y_0 \rightarrow + \infty}  \frac{ y_0^2 (\nu^2-\nu_0^2)^4 f_2(\nu) }{\nu^4 (\nu^2-\nu_0^2)^4 (y_0 (\nu^2-\nu_0^2) +3)^4},
	\end{equation}
	and so
	\begin{align}\nonumber
	\lim_{y_0 \rightarrow + \infty}&\int_{K}  \frac{ y_0^2 (\nu^2-\nu_0^2)^4 f_2(\nu) }{\nu^4 (\nu^2-\nu_0^2)^4 (y_0 (\nu^2-\nu_0^2) +3)^4} \dvol_g \ \mathrm{Vol^*} \\ \nonumber
	& = \lim_{y_0 \rightarrow + \infty} \frac{y_0^{-2}}{18}  \int_{0}^{\epsilon}  \frac{ x f_2(\nu_0) }{  ( x +3/y_0)^4}  \ dx \ \mathrm{Vol^*} + \ldots \\ \nonumber
	& = \frac{f_2(\nu_0) }{18 \times 54}  \mathrm{Vol^*} \\ \nonumber
	& = 4 \nu_0^4  \mathrm{Vol^*}
	\end{align} 
	Now recall that $\nu_0^4\mathrm{Vol^*} = \mathrm{Vol}(\mathbb{S}^3_1) \mathrm{Vol}(Z)=2\pi^2 \mathrm{Vol}(Z)$ and the claimed result follows.
\end{proof}

\begin{remark}
 Notice that $8 \pi^2$ is the energy of the charge-$1$ instanton on $\mathbb{R}^4$, i.e. the bubble. Hence, the above really is an ``energy conservation'' formula, analogous to that of \cites{Lotay2018}, and a generalization of that in \cites{Tian2000} to the infinite energy context.
\end{remark}

\subsection{Fueter sections}\label{subsec:Fueter}

As in the analogous case of $\G_2$-instantons a sequence of $\Spin(7)$-instantons with curvature concentrating along a Cayley submanifold gives rise to a Fueter section of a certain moduli bundle as in \cites{Donaldson2009,Haydys2011,Walpuski2017}. We shall now investigate the particular situation of our set-up and what this Fueter section is.\\

The moduli space $\mathcal{M}$ of framed charge $1$ instantons in $\mathbb{R}^4$ may be written as
$$ \mathcal{M} \cong \left(\Hom_{\mathbb{R}}(\mathbb{C}^2,s^+)  \backslash \lbrace 0 \rbrace  \right) / \mathbb{Z}_2 \times \mathbb{R}^4,$$
where $s^+$ is the $4$-dimensional spin representation of positive chirality and $\Hom_{\mathbb{R}}(\mathbb{C}^2,s^+)  := \Re ( \Hom(\mathbb{C}^2,s^+) )$. This carries an $\SU(2) \times \SO(4)$ action, with $\SU(2)$ acting on the frames in $\Hom_{\mathbb{R}}(\mathbb{C}^2, s^+) $ by precomposition and $\SO(4)$-acting on $\mathbb{R}^4$ by rotations. Then, let $\Fr(NZ)$ the $\SO(4)$-bundle of frames in the normal directions and view $E$ as an $\SU(2)$-bundle. We can define the charge $1$-instanton moduli bundle as
$$\underline{\mathcal{M}} := (E \times \Fr(NZ) ) \times_{\SU(2) \times \SO(4)} \mathcal{M}.$$
Using $\mathcal{S}^+$ to denote the positive spinor bundle of $Z$, the Cayley condition implies that $\mathcal{S}^+ \cong \mathcal{S}^+_{NQ}$. Having this in mind and unravelling through the above identifications we may write
$$\underline{\mathcal{M}} \cong \left(\Hom_{\mathbb{R}}(E,\mathcal{S}^+)  \backslash \lbrace 0 \rbrace  \right) / \mathbb{Z}_2 \times NZ.$$
Furthermore, noticing that in our case both $E \cong NZ$ and also $E \cong \mathcal{S}^+$, so that
$$\underline{\mathcal{M}} \cong \left(\Hom_{\mathbb{R}}(E,E)  \backslash \lbrace 0 \rbrace  \right) / \mathbb{Z}_2 \times E.$$
Now, the bundle $E$ has a quaternionic structure induced by the $\lbrace \frac{s}{24} \ \omega_i \rbrace_{i=1}^3$. Associated to this we have a so called Fueter operator $F$ on sections $s$ of $\underline{\mathcal{M}}$ with values in $\Hom(TZ,V \underline{\mathcal{M}})$, where
$$V\underline{\mathcal{M}} := (E \times \Fr(NZ) ) \times_{\SU(2) \times \SO(4)} T\mathcal{M}.$$
given by
$$F(s)= \nabla s - \sum_{i=1}^3 I_i\nabla_{I_i} s ,$$
where $I_i$ here is seen as acting both on $TZ$ and $NZ$ by using the identification $\Lambda^2_+ Z \cong \Lambda^2_+ NZ$, possible by $Z$ being Cayley. A section $s$ of $\underline{\mathcal{M}}$ is called a Fueter section if $F(s)=0$. In the case we describe above there is an obvious Fueter section which up to scaling is given by
$$s=([\id_E],0),$$
where $[\id_E]$ denotes the equivalence class formed by $\id_E$ and $-\id_E$. 

\begin{remark}
This suggests a possible strategy to find applications of Thomas Walpuski's de-singularization Theorem in \cites{Walpuski2017_Spin}. This is the case when the underlying bundle $E_0$ carrying the initial $\Spin(7)$-instanton restricts to $Z$ as $\mathcal{S}^+$.
\end{remark}

\appendix

\section{Explicit formulas for the metric and its asymptotics}

Recall that $t(r)$ and $\nu(r)$ must lie on a level set of $\nu^4(\nu^2-5t^2)^3$ and for the metric to be (backwards) complete\footnote{It is always forward complete.} it must be a positive level set of this quantity. Indeed, as at $r=0$ we have $t(0)=0$ and $\nu(0)=\nu_0$ parametrizes the solutions which thus satisfy $\nu^4(\nu^2-5t^2)^3=\nu_0^{10}$, or in other words
\begin{equation}\label{eq:t_In_Terms_Of_nu}
t^2 = \frac{\nu^{\frac{10}{3}} - \nu_0^{\frac{10}{3}}}{5 \nu^{\frac{4}{3}}} 
\end{equation}
Furthermore, one can check that $t$ is positive and thus, by the first evolution equation \ref{eq:ode_nu}, $\nu$ is increasing. We may therefore change variables from $r$ to $\nu$ which we may write as
\begin{equation}\label{eq:nu_In_Terms_Of_Int_t}
\nu^2(r)= \nu_0^2 + 6 \int_0^r t(s) ds.
\end{equation}
Somewhat more convenient is to write $dr^2 = (\frac{\partial r}{\partial \nu})^2 d \nu^2$ and use the inverse function theorem together with \ref{eq:t_In_Terms_Of_nu} to write this solely in terms of $\nu$ as
\begin{equation}\label{eq:dr^2}
dr^2 = 	\frac{5}{9} \frac{\nu^{\frac{10}{3}}}{\nu^{\frac{10}{3}} - \nu_0^{\frac{10}{3}}}  d \nu^2.
\end{equation}
From this, the metric can be written as

\begin{eqnarray}
g_{\Theta}=\frac{5}{9}\frac{1}{1-\left(\nu_0\nu^{-1}\right)^{10/3}}d\nu^2 +\frac{1}{5}\frac{1}{1-\left(\nu_0\nu^{-1}\right)^{10/3}}\nu^2\sum_i\eta_i\otimes \eta_i +\frac{s\nu^2}{48} \pi^* g_Z.
\end{eqnarray}


\bibliographystyle{plain}	
\bibliography{refsa}

\begin{bibdiv}
\begin{biblist}

\bib{Berger1955}{article}{
      author={Berger, M.},
       title={Sur les groupes d'holonomie homeg\`ene des vari\'et\'es \`a
  connexion affine et des vari\'et\'es riemanniennes},
        date={1955},
     journal={Bulletin de la Societ\'e Math\'ematiques de France},
      volume={83},
       pages={279\ndash 330},
}

\bib{Boyer2001}{incollection}{
      author={Boyer, C.},
      author={Galicki, K},
       title={$3$-{S}asaki {M}anifolds},
        date={2001},
   booktitle={Surveys in {D}ifferential {G}eometry, vol. 6; {E}ssays on
  {E}instein {M}anifolds},
      volume={6},
   publisher={International Press},
}

\bib{Bonan1966}{article}{
      author={Bonan, E.},
       title={Sur les vari\'et\'es riemanniennes \`a groupe d'holonomie {$G_2$}
  ou {$\Spin(7)$}},
        date={1966},
     journal={C. R. Acad. Sci. Paris},
      volume={262},
       pages={127\ndash 129},
}

\bib{Bryant1987}{article}{
      author={Bryant, Robert~L.},
       title={Metrics with exceptional holonomy},
        date={1987},
        ISSN={0003-486X},
     journal={Ann. of Math. (2)},
      volume={126},
      number={3},
       pages={525\ndash 576},
         url={https://doi.proxy.ufrj.br/10.2307/1971360},
      review={\MR{916718}},
}

\bib{BS89}{article}{
      author={Bryant, Robert~L.},
      author={Salamon, Simon~M.},
       title={On the construction of some complete metrics with special
  holonomy},
        date={1989},
     journal={Duke Math. J.},
      volume={58},
      number={3},
       pages={829\ndash 850},
}

\bib{Corrigan1983}{article}{
      author={Corrigan, E.},
      author={Devchand, C.},
      author={Fairlie, D.~B.},
      author={Nuyts, J.},
       title={First-order equations for gauge fields in spaces of dimension
  greater than four},
        date={1983},
        ISSN={0550-3213},
     journal={Nuclear Phys. B},
      volume={214},
      number={3},
       pages={452\ndash 464},
         url={http://dx.doi.org/10.1016/0550-3213(83)90244-4},
      review={\MR{698892 (84i:81058)}},
}

\bib{Clarke14}{article}{
      author={Clarke, A.},
       title={Instantons on the exceptional holonomy manifolds of {B}ryant and
  {S}alamon},
        date={2014},
        ISSN={0393-0440},
     journal={J. Geom. Phys.},
      volume={82},
       pages={84\ndash 97},
         url={http://dx.doi.org/10.1016/j.geomphys.2014.04.006},
      review={\MR{3206642}},
}

\bib{Donaldson2009}{incollection}{
      author={Donaldson, S.~K.},
      author={Segal, E.~P.},
       title={Gauge theory in higher dimensions, {II}},
        date={2011},
   booktitle={Surveys in differential geometry. {V}olume {XVI}. {G}eometry of
  special holonomy and related topics},
      series={Surv. Differ. Geom.},
      volume={16},
   publisher={Int. Press, Somerville, MA},
       pages={1\ndash 41},
      review={\MR{2893675}},
}

\bib{Donaldson1998}{incollection}{
      author={Donaldson, S.~K.},
      author={Thomas, R.~P.},
       title={Gauge theory in higher dimensions},
        date={1998},
   booktitle={The geometric universe ({O}xford, 1996)},
   publisher={Oxford Univ. Press},
     address={Oxford},
       pages={31\ndash 47},
         url={http://www.ma.ic.ac.uk/~rpwt/skd.pdf},
      review={\MR{MR1634503 (2000a:57085)}},
}

\bib{Fernandez86}{article}{
      author={Fern\'andez, M.},
       title={A classification of {R}iemannian manifolds with structure group
  {$\Spin(7)$}},
        date={1986},
     journal={Ann. Mat. Pura Appl.},
      volume={143},
      number={4},
       pages={101\ndash 122},
}

\bib{Foscolo2015}{article}{
      author={Foscolo, L.},
      author={Haskins, M.},
       title={New {${\rm G}_2$}-holonomy cones and exotic nearly {K}\"ahler
  structures on {$S^6$} and {$S^3\times S^3$}},
        date={2017},
        ISSN={0003-486X},
     journal={Ann. of Math. (2)},
      volume={185},
      number={1},
       pages={59\ndash 130},
         url={http://dx.doi.org/10.4007/annals.2017.185.1.2},
      review={\MR{3583352}},
}

\bib{FKMS1997}{article}{
      author={Friedrich, Th.},
      author={Kath, I.},
      author={Moroianu, A.},
      author={Semmelmann, U.},
       title={On nearly parallel {$G_2$}-structures},
        date={1997},
     journal={J. Geom. Phys.},
      volume={23},
       pages={259\ndash 286},
}

\bib{Gibbons1990}{article}{
      author={Gibbons, G.W.},
      author={Page, D.N.},
      author={Pope, C.N.},
       title={Einstein metrics on {$S^3$}, $\mathbb{R}^3$ and $\mathbb{R}^4$
  bundles},
        date={1990},
     journal={Comm. Math. Phys.},
      volume={127},
}

\bib{Haydys2011}{article}{
      author={Haydys, A.},
       title={Gauge theory, calibrated geometry and harmonic spinors},
        date={2012},
        ISSN={0024-6107},
     journal={J. Lond. Math. Soc. (2)},
      volume={86},
      number={2},
       pages={482\ndash 498},
         url={http://dx.doi.org/10.1112/jlms/jds008},
      review={\MR{2980921}},
}

\bib{Harland2012}{article}{
      author={Harland, D.},
      author={N\"olle, Ch.},
       title={Instantons and {K}illing spinors},
        date={2012},
     journal={J. High Energy Phys.},
      volume={82},
}

\bib{Joyce1999}{article}{
      author={Joyce, Dominic},
       title={A new construction of compact 8-manifolds with holonomy {${\rm
  Spin}(7)$}},
        date={1999},
     journal={J. Diff. Geom.},
      volume={53},
      number={1},
       pages={89\ndash 130},
         url={https://doi.org/10.4310/jdg/1214425448},
}

\bib{Lewis1998}{thesis}{
      author={Lewis, C.},
       title={{$\Spin(7)$} instantons},
        type={Ph.D. Thesis},
        date={1998},
}

\bib{Lotay2017}{incollection}{
      author={Lotay, J.~D.},
      author={Madsen, T.~B.},
       title={Instantons and special geometry},
        date={2017},
   booktitle={Special {M}etrics and {G}roup {A}ctions in {G}eometry. {S}pringer
  {IN}d{AM} {S}eries, vol 23.},
      editor={Chiossi, S. et~al.},
   publisher={Springer, Cham},
}

\bib{Lotay2018}{article}{
      author={Lotay, Jason~D.},
      author={Oliveira, Gon{\c c}alo},
       title={{$SU(2)^2$}-invariant {$G_2$} instantons},
        date={2018},
     journal={Mathematische Annalen},
      volume={371},
      number={1-2},
       pages={961\ndash 1011},
}

\bib{Malgrange1974}{article}{
      author={Malgrange, B.},
       title={Sur les points singuliers des equations diff\'erentielles},
        date={1974},
        ISSN={0013-8584},
     journal={Enseignement Math. (2)},
      volume={20},
       pages={147\ndash 176},
      review={\MR{0368074}},
}

\bib{Simons1962}{article}{
      author={Simons, J.},
       title={On the transitivity of holonomy systems},
        date={1962},
     journal={Ann. Math.},
      volume={76},
       pages={213\ndash 234},
}

\bib{Tanaka2012}{article}{
      author={Tanaka, Y.},
       title={A construction of {$\Spin(7)$}-instantons},
        date={2012},
     journal={Ann. Global Anal. Geom.},
      volume={42},
      number={4},
       pages={495\ndash 521},
}

\bib{Tian2000}{article}{
      author={Tian, G.},
       title={Gauge theory and calibrated geometry. {I}},
        date={2000},
        ISSN={0003-486X},
     journal={Ann. of Math. (2)},
      volume={151},
      number={1},
       pages={193\ndash 268},
         url={http://dx.doi.org/10.2307/121116},
      review={\MR{MR1745014 (2000m:53074)}},
}

\bib{Tao2004}{article}{
      author={Tao, T.},
      author={Tian, G.},
       title={A singularity removal theorem for {Y}ang--{M}ills fields in
  higher dimensions},
        date={2004},
        ISSN={0894-0347},
     journal={J. Amer. Math. Soc.},
      volume={17},
      number={3},
       pages={557\ndash 593 (electronic)},
         url={http://dx.doi.org/10.1090/S0894-0347-04-00457-6},
      review={\MR{MR2053951 (2005f:58013)}},
}

\bib{Walpuski2017}{article}{
      author={Walpuski, T.},
       title={{${\rm G}_2$--instantons, associative submanifolds and {F}ueter
  sections}},
        date={2017},
     journal={Comm. Anal. Geom.},
      volume={25},
      number={4},
       pages={847\ndash 893},
}

\bib{Walpuski2017_Spin}{article}{
      author={Walpuski, Thomas},
       title={{$\Spin(7)$}-instantons, {C}ayley submanifolds and {F}ueter
  sections},
        date={2017},
     journal={Communications in Mathematical Physics},
      volume={352},
      number={1},
       pages={1\ndash 36},
}

\bib{Ward1984}{article}{
      author={Ward, R.~S.},
       title={Completely solvable gauge-field equations in dimension greater
  than four},
        date={1984},
     journal={Nuclear Phys. B},
      volume={236},
      number={2},
       pages={381\ndash 396},
}

\end{biblist}
\end{bibdiv}

\end{document}